\documentclass[12pt,twoside]{amsart}
\usepackage[latin1]{inputenc}
\usepackage{amsmath, amsthm, amscd, amsfonts, amssymb, graphicx}
\usepackage[bookmarksnumbered, plainpages]{hyperref}

\newtheorem{thm}{Theorem}[section]

\newtheorem{lem}[thm]{Lemma}
\newtheorem{prop}[thm]{Proposition}

\numberwithin{equation}{section}

\begin{document}

\title{{Several new Witten rigidity theorems for spin$^c$ manifolds}}

\author{ Jianyun Guan, Kefeng Liu and Yong Wang*\\
 }

\date{}

\thanks{{\scriptsize
\hskip -0.4 true cm \textit{2010 Mathematics Subject Classification:}
58C20; 57R20; 53C80.
\newline \textit{* Corresponding author.}}}

\maketitle

\begin{abstract}
Using Liu's modular invariance method and its odd-dimensional extension by Han and Yu, we establish new Witten rigidity theorems for the generalized Witten genus of twisted Dirac operators on even-dimensional spin$^c$ manifolds and twisted Toeplitz operators on odd-dimensional spin$^c$ manifolds with circle actions.
\end{abstract}

\vskip 0.2 true cm

\section{Introduction}

The study of rigidity properties for fundamental geometric operators on various manifolds occupies a central place at the intersection of differential geometry and mathematical physics. This line of inquiry gained substantial momentum in 1982, when Witten, motivated by physical intuition, established the rigidity of the twisted Dirac operator $\mathcal{D}\otimes TX$ on compact homogeneous spin manifolds. In 1988, he further derived a series of twisted Dirac operators on the free loop space $LM$ of a spin manifold $M$ \cite{WE1}. Notably, in that work, Witten discovered-to his surprise-that the elliptic genus, constructed topologically by Landweber and Stone \cite{LS} and by Ochanine \cite{OS}, equals the index of one of these operators. Inspired by physical considerations, Witten conjectured that these elliptic operators should be rigid. This conjecture was first confirmed in 1989 through independent work by Taubes \cite{TH} and Bott and Taubes \cite{BT}. Subsequent developments included proofs of the conjecture for almost complex manifolds by Hirzebruch \cite{HF} and Krichever \cite{KM}.

A significant simplification and generalization came in 1996, when Liu employed modular invariance to provide a unified proof of the Witten conjecture, together with several broad extensions and new vanishing theorems \cite{Li2,Li3}. To further generalize, in 2000, Dessai established rigidity and vanishing theorems for the spin$^c$ case. Liu and Ma \cite{Li4,Li5} and Liu, Ma, and Zhang \cite{Li6,Li7} later extended the rigidity and vanishing theorems to families, both at the level of the equivariant Chern character and in equivariant K-theory.

In a parallel research direction, the study of rigidity for Toeplitz-type operators emerged. In 2009, Liu and Wang established rigidity for twisted Toeplitz operators associated with Witten bundles under the assumption that the fixed-point sets of the group action are one-dimensional, marking the first investigation of Witten-type rigidity for such operators \cite{Li2,Li3}. This result was substantially expanded in 2015 by Han and Yu, who proved rigidity and vanishing properties for twisted Toeplitz operators allowing fixed-point sets of arbitrary dimension \cite{HY}.

In \cite{GW,GW1}, we constructed some new $SL_2(\mathbb{Z})$-modular forms, $\Gamma^0(2)$-modular forms, and $\Gamma_0(2)$-modular forms on spin manifolds using the modular invariance of characteristic forms. Moreover, we obtained divisibility results for the indices of twisted Dirac operators and twisted Toeplitz operators on spin manifolds. Furthermore, using Liu's method and the method of Han and Yu, we proved that these new modular forms are rigid \cite{GLW}. In \cite{CHZ}, Chen, Han, and Zhang constructed a generalized Witten genus for $4k$-dimensional spin$^c$ manifolds and a generalized Witten genus for $(4k+2)$-dimensional spin$^c$ manifolds. Motivated by \cite{GLW} and \cite{CHZ}, the purpose of this paper is to construct new generalized Witten genera on spin$^c$ manifolds and then to explore whether the relevant Dirac operators are rigid by applying Liu's method and the method of Han and Yu.

The structure of this paper is as follows. In Section 2, we introduce definitions and basic concepts that will be used throughout. In Section 3, we prove Witten-type rigidity for twisted Dirac operators on even-dimensional spin$^c$ manifolds. Finally, in Section 4, we prove Witten-type rigidity for twisted Toeplitz operators on odd-dimensional spin$^c$ manifolds.

\section{Characteristic Forms and Modular Forms}
\quad The purpose of this section is to review the necessary background on characteristic forms and modular forms that we will use.

 \noindent {\bf  2.1 Characteristic forms }\\
 \indent Let $M$ be a Riemannian manifold.
 Let $\nabla^{ TM}$ be the associated Levi-Civita connection on $TM$
 and $R^{TM}=(\nabla^{TM})^2$ be the curvature of $\nabla^{ TM}$.
 Let $\widehat{A}(TM,\nabla^{ TM})$ and $\widehat{L}(TM,\nabla^{ TM})$
 be the Hirzebruch characteristic forms defined respectively by (cf. \cite{Zh})
\begin{equation}
   \widehat{A}(TM,\nabla^{ TM})={\rm
det}^{\frac{1}{2}}\left(\frac{\frac{\sqrt{-1}}{4\pi}R^{TM}}{{\rm
sinh}(\frac{\sqrt{-1}}{4\pi}R^{TM})}\right),
\end{equation}
 \begin{equation}
     \widehat{L}(TM,\nabla^{ TM})={\rm
 det}^{\frac{1}{2}}\left(\frac{\frac{\sqrt{-1}}{2\pi}R^{TM}}{{\rm
 tanh}(\frac{\sqrt{-1}}{4\pi}R^{TM})}\right).
 \end{equation}
   Let $E$, $F$ be two Hermitian vector bundles over $M$ carrying
   Hermitian connections $\nabla^E,\nabla^F$ respectively. Let
   $R^E=(\nabla^E)^2$ (resp. $R^F=(\nabla^F)^2$) be the curvature of
   $\nabla^E$ (resp. $\nabla^F$). If we set the formal difference
   $G=E-F$, then $G$ carries an induced Hermitian connection
   $\nabla^G$ in an obvious sense. We define the associated Chern
   character form as
   \begin{equation}
       {\rm ch}(G,\nabla^G)={\rm tr}\left[{\rm
   exp}(\frac{\sqrt{-1}}{2\pi}R^E)\right]-{\rm tr}\left[{\rm
   exp}(\frac{\sqrt{-1}}{2\pi}R^F)\right].
   \end{equation}
   For any complex number $t$, let
   $$\wedge_t(E)={\bf C}|_M+tE+t^2\wedge^2(E)+\cdots,~S_t(E)={\bf
   C}|_M+tE+t^2S^2(E)+\cdots$$
   denote respectively the total exterior and symmetric powers of
   $E$, which live in $K(M)[[t]].$ The following relations between
   these operations hold,
   \begin{equation}
       S_t(E)=\frac{1}{\wedge_{-t}(E)},~\wedge_t(E-F)=\frac{\wedge_t(E)}{\wedge_t(F)}.
   \end{equation}
   Moreover, if $\{\omega_i\},\{\omega_j'\}$ are formal Chern roots
   for Hermitian vector bundles $E,F$ respectively, then
   \begin{equation}
       {\rm ch}(\wedge_t(E))=\prod_i(1+e^{\omega_i}t).
   \end{equation}
   Then we have the following formulas for Chern character forms,
   \begin{equation}
       {\rm ch}(S_t(E))=\frac{1}{\prod_i(1-e^{\omega_i}t)},~
{\rm ch}(\wedge_t(E-F))=\frac{\prod_i(1+e^{\omega_i}t)}{\prod_j(1+e^{\omega_j'}t)}.
   \end{equation}
\indent If $W$ is a real Euclidean vector bundle over $M$ carrying a
Euclidean connection $\nabla^W$, then its complexification $W_{\bf
C}=W\otimes {\bf C}$ is a complex vector bundle over $M$ carrying a
canonical induced Hermitian metric from that of $W$, as well as a
Hermitian connection $\nabla^{W_{\bf C}}$ induced from $\nabla^W$.
If $E$ is a vector bundle (complex or real) over $M$, set
$\widetilde{E}=E-{\rm dim}E$ in $K(M)$ or $KO(M)$.\\

\noindent{\bf 2.2 Some properties of Jacobi theta functions and modular forms}\\
   \indent We first recall the four Jacobi theta functions defined as follows (cf. \cite{Ch}):
   \begin{equation}
      \theta(v,\tau)=2q^{\frac{1}{8}}{\rm sin}(\pi
   v)\prod_{j=1}^{\infty}[(1-q^j)(1-e^{2\pi\sqrt{-1}v}q^j)(1-e^{-2\pi\sqrt{-1}v}q^j)],
   \end{equation}
\begin{equation}
    \theta_1(v,\tau)=2q^{\frac{1}{8}}{\rm cos}(\pi
   v)\prod_{j=1}^{\infty}[(1-q^j)(1+e^{2\pi\sqrt{-1}v}q^j)(1+e^{-2\pi\sqrt{-1}v}q^j)],
\end{equation}
\begin{equation}
    \theta_2(v,\tau)=\prod_{j=1}^{\infty}[(1-q^j)(1-e^{2\pi\sqrt{-1}v}q^{j-\frac{1}{2}})
(1-e^{-2\pi\sqrt{-1}v}q^{j-\frac{1}{2}})],
\end{equation}
\begin{equation}
   \theta_3(v,\tau)=\prod_{j=1}^{\infty}[(1-q^j)(1+e^{2\pi\sqrt{-1}v}q^{j-\frac{1}{2}})
(1+e^{-2\pi\sqrt{-1}v}q^{j-\frac{1}{2}})],
\end{equation}
 \noindent
where $q=e^{2\pi\sqrt{-1}\tau}$ with $\tau\in\textbf{H}$, the upper
half complex plane. Let
\begin{equation}
    \theta'(0,\tau)=\frac{\partial\theta(v,\tau)}{\partial v}|_{v=0}.
\end{equation} \noindent Then the following Jacobi identity
(cf. \cite{Ch}) holds,
\begin{equation}   \theta'(0,\tau)=\pi\theta_1(0,\tau)\theta_2(0,\tau)\theta_3(0,\tau).
\end{equation}
\noindent Denote $$SL_2(\mathbb{Z})=\left\{\left(\begin{array}{cc}
\ a & b  \\
 c  & d
\end{array}\right)\mid a,b,c,d \in {\bf Z},~ad-bc=1\right\}$$
the modular group. Let $S=\left(\begin{array}{cc}
\ 0 & -1  \\
 1  & 0
\end{array}\right),~T=\left(\begin{array}{cc}
\ 1 &  1 \\
 0  & 1
\end{array}\right)$ be the two generators of $SL_2(\mathbb{Z})$. They
act on $\textbf{H}$ by $S\tau=-\frac{1}{\tau},~T\tau=\tau+1$.

\noindent
 \noindent {\bf Definition 2.1} A modular form over $\Gamma$, a
 subgroup of $SL_2(\mathbb{Z})$, is a holomorphic function $f(\tau)$ on
 $\textbf{H}$ such that
 \begin{equation}
    f(g\tau):=f\left(\frac{a\tau+b}{c\tau+d}\right)=\chi(g)(c\tau+d)^kf(\tau),
 ~~\forall g=\left(\begin{array}{cc}
\ a & b  \\
 c & d
\end{array}\right)\in\Gamma,
 \end{equation}
\noindent where $\chi:\Gamma\rightarrow {\bf C}^{\star}$ is a
character of $\Gamma$. $k$ is called the weight of $f$.\\
Let $$\Gamma_0(2)=\left\{\left(\begin{array}{cc}
\ a & b  \\
 c  & d
\end{array}\right)\in SL_2(\mathbb{Z})\mid c\equiv 0~({\rm
mod}~2)\right\},$$
$$\Gamma^0(2)=\left\{\left(\begin{array}{cc}
\ a & b  \\
 c  & d
\end{array}\right)\in SL_2(\mathbb{Z})\mid b\equiv 0~({\rm
mod}~2)\right\},$$
be the two modular subgroups of $SL_2(\mathbb{Z})$.
It is known that the generators of $\Gamma_0(2)$ are $T,~ST^2ST$,
the generators of $\Gamma^0(2)$ are $STS,~T^2STS$ (cf. \cite{Ch}).

\section{Twisted Dirac operator and Witten rigidity theorem in even dimensional Spin$^c$ manifolds}
Let $M$ be a closed oriented ${\rm spin^{c}}$-manifold and $L$ be the complex line bundle associated to the given ${\rm spin^{c}}$ structure on $M.$ Denote by $c=c_1(L)$ the first Chern class of $L.$ Also, we use $L_{\bf{R}}$ for the notation of $L,$ when it is viewed as an oriented real plane bundle.
Let $\Theta(T_{\mathbf{C}}M,L_{\bf{R}}\otimes\bf{C})$ be the virtual complex vector bundle over $M$ defined by
\begin{equation}
    \begin{split}
        \Theta(T_{\mathbf{C}}M,L_{\bf{R}}\otimes\mathbf{C})=&\bigotimes _{n=1}^{\infty}S_{q^n}(\widetilde{T_{\mathbf{C}}M})\otimes
\bigotimes _{m=1}^{\infty}\wedge_{q^m}(\widetilde{L_{\bf{R}}\otimes\mathbf{C}})\\
&\otimes
\bigotimes _{r=1}^{\infty}\wedge_{-q^{r-\frac{1}{2}}}(\widetilde{L_{\bf{R}}\otimes\mathbf{C}})\otimes
\bigotimes _{s=1}^{\infty}\wedge_{q^{s-\frac{1}{2}}}(\widetilde{L_{\bf{R}}\otimes\mathbf{C}}).\nonumber
    \end{split}
\end{equation}
 
Let $V$ be a rank $2l$ real spin vector bundle on $M$. Moreover, $V_{\mathbf{C}}=V\otimes\mathbf{C}.$  Set
\begin{equation}
Q_1(V_{\mathbf{C}})=\Delta(V)\otimes\bigotimes_{n=1}^{\infty}\Lambda_{q^n}(\widetilde{V_{\mathbf{C}}}),
\end{equation}
\begin{equation}
Q_2(V_{\mathbf{C}})=\bigotimes_{n=1}^{\infty}\Lambda_{-q^{n-\frac{1}{2}}}(\widetilde{V_{\mathbf{C}}}),
\end{equation}
\begin{equation}
Q_3(V_{\mathbf{C}})=\bigotimes_{n=1}^{\infty}\Lambda_{q^{n-\frac{1}{2}}}(\widetilde{V_{\mathbf{C}}}).
\end{equation}

Let $g^{TM}$ be a Riemannian metric on $M$. Let $\nabla^{TM}$ be the Levi-Civita connection associated to $g^{TM}$. Let $g^{T_{\mathbf{C}}M}$ and $\nabla^{T_{\mathbf{C}}M}$ be the induced Hermitian metric and Hermitian connection on $T_{\mathbf{C}}M$. Let $h^L$ be a Hermitian metric on $L$ and $\nabla^L$ be a Hermitian connection. Let $h^{L_{\mathbf{R}}}$ and $\nabla^{L_{\mathbf{R}}}$ be the induced Euclidean metric and connection on $L_{\mathbf{R}}$. Then $\nabla^{TM}$ and $\nabla^L$ induce connections $\nabla^{\Theta(T_{\mathbf{C}}M,L_{\bf{R}}\otimes\mathbf{C})}$ on $\Theta(T_{\mathbf{C}}M,L_{\bf{R}}\otimes\mathbf{C})$. Let ${\rm dim}M=4k$ and $u=-\frac{\sqrt{-1}}{2\pi}c.$ Define the generalized Witten forms
\begin{equation}
    W_1(M,L,V)=\left\{\widehat{A}(TM){\rm exp}(\frac{c}{2}){\rm ch}(\Theta(T_{\mathbf{C}}M,L_{\bf{R}}\otimes\mathbf{C})){\rm ch}(Q_1(V_{\mathbf{C}}))\right\}^{4k},
\end{equation}
\begin{equation}
    W_2(M,L,V)=\left\{\widehat{A}(TM){\rm exp}(\frac{c}{2}){\rm ch}(\Theta(T_{\mathbf{C}}M,L_{\bf{R}}\otimes\mathbf{C})){\rm ch}(Q_2(V_{\mathbf{C}}))\right\}^{4k},
\end{equation}
\begin{equation}
    W_3(M,L,V)=\left\{\widehat{A}(TM){\rm exp}(\frac{c}{2}){\rm ch}(\Theta(T_{\mathbf{C}}M,L_{\bf{R}}\otimes\mathbf{C})){\rm ch}(Q_3(V_{\mathbf{C}}))\right\}^{4k}.
\end{equation}

Define the generalized Witten genus
\begin{equation}
\mathcal{W}_\lambda(M,L,V):=\int_{M^{4k}}W_\lambda(M,L,V), \ \  \lambda=1,2,3.
\end{equation}

Let $$S_c(TM)=S_{c,+}(TM)\oplus S_c(TM)=S_{c,-}(TM)$$ denote the bundle of spinors associated to the spin$^c$ structure, $(TM,g^{TM})$ and $(L,h^L)$. Then $S_c(TM)$ carries induced Hermitian metric and connection preserving the above $\mathbf{Z}_2$-grading. Let $$\mathcal{D}_c:\Gamma(S_{c,\pm}(TM))\to\Gamma(S_{c,\mp}(TM))$$ denote the induced spin$^c$ Dirac operators (cf. \cite{LM}). Then $$\mathcal{D}_c\otimes\Theta(T_{\mathbf{C}}M,L_{\bf{R}}\otimes\mathbf{C})\otimes Q_\lambda(V_{\mathbf{C}}) $$ which maps between the smooth section spaces,
$$\Gamma(S_{c,\pm}(TM)\otimes\Theta(T_{\mathbf{C}}M,L_{\bf{R}}\otimes\mathbf{C})\otimes Q_\lambda(V_{\mathbf{C}}))\to\Gamma(S_{c,\mp}(TM)\otimes\Theta(T_{\mathbf{C}}M,L_{\bf{R}}\otimes\mathbf{C})\otimes Q_\lambda(V_{\mathbf{C}}))$$ denote the corresponding twisted spin$^c$ Dirac operators.

Let $M$ be a closed smooth spin$^c$ Riemannian manifold which admits a circle action. Without loss of generality, we may assume that $M$ admits an $S^1$-action that lifts to an action on $L$. Moreover, we assume that this action preserves the given spin$^c$-structure associated to $(M, L)$, as well as the metrics and connections involved. Let $c_1(L)_{S^1}$ be the first equivariant Chern class of $L$. 

Let $V$ be an oriented real rank-$2\bar{l}$ vector bundle on a manifold $M$, equipped with an $S^1$-action that restricts on each fiber of $V$ over $M$ to a linear action preserving that fiber. Then the associated twisted Dirac operator is $S^1$-equivariant.

Next we assume $g=e^{2\pi \mathbf{i}t}\in S^1$ be a generator of the action group. Let $F$ denote the fixed submanifold of the circle action on $M$ and $e^{2\pi \mathbf{i}t}$ acts on $L$
by $e^{2\pi\mathbf{i}\sigma t}$. In general, $F$ is not connected. We fix a connected component $F_{\alpha}$. Let $\mathbf{N}$ denote the normal bundle to $F_{\alpha}$ in $M$, which can be identified as the orthogonal complement of $TF_{\alpha}$ in $TM|_{F_{\alpha}}$. Let $\dim F_\alpha=2s$ and $\dim\mathbf{N}=2\bar{r}$ with $\bar{r},s\in\mathbb{Z}\setminus0$. Then we have the following $S^1$-equivariant decomposition when restricted to $F_{\alpha}$,
\begin{equation}
  TM|_{F_{\alpha}}=\mathbf{N}_1\oplus\cdots\oplus \mathbf{N}_{r}\oplus TF_{\alpha},
\end{equation}
where each $\mathbf{N}_{\beta},$~$\beta=1,\cdots,r$ is a complex vector bundle, and $g$ acts on $\mathbf{N}_{\beta}$
by $e^{2\pi \mathbf{i}m_{\beta}t}$. We assume $\{2\pi\mathbf{i}x_{\beta}^{i_\beta}, \ 1\leq i_\beta\leq\dim\mathbf{N}_\beta\}$ be the Chern roots of $\mathbf{N}_\beta$. Let $\{\pm2\pi\mathbf{i}y_j, \ 1\leq j\leq s\}$ be the Chern roots of $TF_{\alpha}\otimes\mathbf{C}$.
Similarly, let
\begin{equation}
  V|_{F_\alpha}=V_1\oplus\cdots\oplus V_l\oplus V^{\mathbb{R}}_0
\end{equation}
be the $S^1$-equivariant decomposition of the restrictions of $V$ over $F_\alpha$, where $V_\nu,~1\leq\nu\leq l$ is a complex vector bundle and $V^{\mathbb{R}}_0$ is the real subbundle of $V|_{F_\alpha}$ such that $S^1$ acts as identity. Assume that $g$ acts on $V_\nu$ by $e^{2\pi \mathbf{i}n_\nu t},~1\leq\nu\leq l$ and let $n_0=0$. Let $\{2\pi\mathbf{i}z_{\nu}^{h_\nu},~1\leq h_\nu\leq \dim V_\nu\}$ be the Chern roots of $V_\nu$. Let $\{\pm2\pi\mathbf{i}z^{h_0}_{0},~1\leq h_0\leq \dim V_0^{\mathbb{R}}\}$ be the Chern roots of $V^{\mathbb{R}}_0\otimes\mathbf{C}$.

We suppose $p_1(\cdot)_{S^1}$ denote the first $S^1$-equivariant Pontrjagin class, then we have the following rigidity theorems.
\begin{thm}
  For a $4k$-dimensional connected spin$^c$ manifold with a non-trivial $S^1$-action, if $3p_1(L)_{S^1}+p_1(V)_{S^1}=p_1(TM)_{S^1}$, then the twisted Dirac operators $$\mathcal{D}_c\otimes\Theta(T_{\mathbf{C}}M,L_{\bf{R}}\otimes\mathbf{C})\otimes Q_\lambda(V_{\mathbf{C}}),~\lambda=1,2,3$$ are rigid.
\end{thm}

To prove this rigidity theorem, we first calculate the corresponding Lefschetz numbers.
\begin{prop}\label{pro3.2}
 The Lefschetz numbers of $\mathcal{D}_c\otimes\Theta(T_{\mathbf{C}}M,L_{\bf{R}}\otimes\mathbf{C})\otimes Q_\lambda(V_{\mathbf{C}}),~\lambda=1,2,3$ are
\begin{equation}\label{eq:3.10}
  \begin{split}
  \mathcal{L}_1(g;\tau)=&2^{\bar{l}-\bar{r}}\left(\frac{-\mathbf{i}}{\pi}\right)^{\bar{r}}
  \sum_{\alpha}\int_{F_\alpha}\left(\prod^{s}_{j=1}
  y_j\frac{\theta'(0,\tau)}{\theta(y_j,\tau)}\prod^{r}_{\beta=1}
  \prod^{\dim\mathbf{N}_\beta}_{i_\beta=1}\frac{\theta'(0,\tau)}{\theta(x_{\beta}^{i_\beta}+
  m_{\beta}t,\tau)}\right)\\
  &\cdot\frac{\theta_1(u+\sigma t,\tau)\theta_2(u+\sigma t,\tau)\theta_3(u+\sigma t,\tau)}
  {\theta_1(0,\tau)\theta_2(0,\tau)
  \theta_3(0,\tau)}\cdot\prod_{\nu=0}^l\prod_{h_\nu=1}^{\dim V_\nu}\frac{\theta_1(z_\nu^{h_\nu}+n_\nu t,\tau)}{\theta_1(0,\tau)},
  \end{split}
\end{equation}
\begin{equation}
  \begin{split}
  \mathcal{L}_2(g;\tau)=&-\left(\frac{\mathbf{i}}{2\pi}\right)^{\bar{r}}
  \sum_{\alpha}\int_{F_\alpha}\left(\prod^{s}_{j=1}
  y_j\frac{\theta'(0,\tau)}{\theta(y_j,\tau)}\prod^{r}_{\beta=1}
  \prod^{\dim\mathbf{N}_\beta}_{i_\beta=1}\frac{\theta'(0,\tau)}{\theta(x_{\beta}^{i_\beta}+
  m_{\beta}t,\tau)}\right)\\
  &\cdot\frac{\theta_1(u+\sigma t,\tau)\theta_2(u+\sigma t,\tau)\theta_3(u+\sigma t,\tau)}
  {\theta_1(0,\tau)\theta_2(0,\tau)
  \theta_3(0,\tau)}\cdot\prod_{\nu=0}^l\prod_{h_\nu=1}^{\dim V_\nu}\frac{\theta_2(z_\nu^{h_\nu}+n_\nu t,\tau)}{\theta_2(0,\tau)},
  \end{split}
\end{equation}
\begin{equation}\label{eq:3.12}
  \begin{split}
  \mathcal{L}_3(g;\tau)=&-\left(\frac{\mathbf{i}}{2\pi}\right)^{\bar{r}}
  \sum_{\alpha}\int_{F_\alpha}\left(\prod^{s}_{j=1}
  y_j\frac{\theta'(0,\tau)}{\theta(y_j,\tau)}\prod^{r}_{\beta=1}
  \prod^{\dim\mathbf{N}_\beta}_{i_\beta=1}\frac{\theta'(0,\tau)}{\theta(x_{\beta}^{i_\beta}+
  m_{\beta}t,\tau)}\right)\\
  &\cdot\frac{\theta_1(u+\sigma t,\tau)\theta_2(u+\sigma t,\tau)\theta_3(u+\sigma t,\tau)}
  {\theta_1(0,\tau)\theta_2(0,\tau)
  \theta_3(0,\tau)}\cdot\prod_{\nu=0}^l\prod_{h_\nu=1}^{\dim V_\nu}\frac{\theta_3(z_\nu^{h_\nu}+n_\nu t,\tau)}{\theta_3(0,\tau)}.
  \end{split}
\end{equation}
\end{prop}
\begin{proof}
 By the Lefschetz fixed point formula (cf. \cite{Li5})
 \begin{equation}\label{eq:3.13}
   \mathcal{L}(g)=\sum_i\int_{F_i}\hat{A}(TF_i)\left[Pf\left(2\sinh\left(\frac{\Omega^\perp}{4\pi}+
   \sqrt{-1}\frac{\Theta_j}{2}\right)\right)\right]^{-1}e^{\frac{c}{2}+\pi\mathbf{i}\sigma t}{\rm ch}_g(E).
 \end{equation}
 We can calculate directly
\begin{equation}
 \begin{split}
  \mathcal{L}_\lambda(g,\tau)=&\prod_{j=1}^{s}\frac{\pi y_j}{\sin \pi y_j}\prod^{r}_{\beta=1}\prod^{\dim\mathbf{N}_\beta}_{i_\beta=1}\frac{1}{2\mathbf{i}\sin(\pi
  (x_{\beta}^{i_\beta}+m_{\beta}t))}e^{\pi\sqrt{-1}(u+\sigma t)}\\
  &\cdot{\rm ch}_g(\Theta(T_{\mathbf{C}}M,L_{\bf{R}}\otimes\mathbf{C})
  ){\rm ch}_g(Q_\lambda(V_{\mathbf{C}})),
 \end{split}  
\end{equation}
\begin{equation}
 \begin{split}
  e^{\pi\sqrt{-1}(u+\sigma t)}
  &{\rm ch}_g(\Theta(T_{\mathbf{C}}M,L_{\bf{R}}\otimes\mathbf{C})
  )\\
  =&\prod_{j=1}^{s}\frac{\sin(\pi y_j)}{\pi}\left(\frac{\theta'(0,\tau)}{\theta(y_j,\tau)}\right)\prod^{r}_{\beta=1}
  \prod^{\dim\mathbf{N}_\beta}_{i_\beta=1}\frac{\sin(\pi
  (x_{\beta}^{i_\beta}+m_{\beta}t))}{\pi}\\
  &\cdot\left(\frac{\theta'(0,\tau)}{\theta(x_{\beta}^{i_\beta}+
  m_{\beta}t,\tau)}\right)
  \cdot\frac{\theta_1(u+\sigma t,\tau)\theta_2(u+\sigma t,\tau)\theta_3(u+\sigma t,\tau)}
  {\theta_1(0,\tau)\theta_2(0,\tau)
  \theta_3(0,\tau)},
 \end{split}
\end{equation}
\begin{equation}
  {\rm ch}_g\left(\Delta(V)\otimes
  \bigotimes_{n=1}^{\infty}\Lambda_{q^n}(\widetilde{V_{\mathbf{C}}})\right)=2^{\bar{l}}
  \prod_{\nu=0}^l\prod_{h_\nu=1}^{\dim V_\nu}\frac{\theta_1(z_\nu^{h_\nu}+n_\nu t,\tau)}{\theta_1(0,\tau)},
\end{equation}
\begin{equation}
 {\rm ch}_g\left(
\bigotimes_{n=1}^{\infty}\Lambda_{-q^{n-\frac{1}{2}}}(\widetilde{V_{\mathbf{C}}})\right) =\prod_{\nu=0}^l\prod_{h_\nu=1}^{\dim V_\nu}\frac{\theta_2(z_\nu^{h_\nu}+n_\nu t,\tau)}{\theta_2(0,\tau)},
\end{equation}
\begin{equation}\label{eq:3.18}
  {\rm ch}_g\left(
  \bigotimes_{n=1}^{\infty}\Lambda_{q^{n-\frac{1}{2}}}(\widetilde{V_{\mathbf{C}}})\right)
  =\prod_{\nu=0}^l\prod_{h_\nu=1}^{\dim V_\nu}\frac{\theta_3(z_\nu^{h_\nu}+n_\nu t,\tau)}{\theta_3(0,\tau)}.
\end{equation}
Thus, Proposition \ref{pro3.2} follows from \eqref{eq:3.13}--\eqref{eq:3.18}.
\end{proof}

In what follows, we view the above expressions as defining functions $\mathcal{L}_\lambda'(t,\tau)$ such that $\mathcal{L}_\lambda'(t,\tau)=\mathcal{L}_\lambda(g;\tau)$.
The expressions for $\mathcal{L}_\lambda'(t,\tau)$ involve only complex-differentiable functions, so we can extend their domain to every complex number $t$ and choice of $\tau$ in the open upper half-plane, i.e., $\mathcal{L}_\lambda'(t,\tau)$ where the expression exists in $\mathbf{C}\times\mathbf{H}$. The Witten rigidity theorems are equivalent to the statement that these $\mathcal{L}_\lambda'(t,\tau)$ are independent of $t$.  We then have the following lemma.
\begin{lem}\label{le3.3}
Let $(t,\tau)\in\mathbf{C}\times\mathbf{H}$ be in the domain of $\mathcal{L}_\lambda'(t,\tau)$.\\
(1) Then $\mathcal{L}_\lambda'(t,\tau)=\mathcal{L}_\lambda'(t+a,\tau)$ for any $a\in 2\mathbb{Z}$.\\
(2) If $3p_1(L)_{S^1}+p_1(V)_{S^1}=p_1(TM)_{S^1}$, then $\mathcal{L}_\lambda'(t,\tau)=\mathcal{L}_\lambda'(t+a\tau,\tau)$ for any $a\in 2\mathbb{Z}$.
\end{lem}
\begin{proof}
  We have the following transformation formulas for theta-functions \cite{Ch}:
  \begin{equation}
    \theta(t+1,\tau)=-\theta(t,\tau), \ \ \theta(t+\tau,\tau)=-q^{-1/2}e^{-2\pi\mathbf{i}t}\theta(t,\tau),\nonumber
  \end{equation}
  \begin{equation}
    \theta_1(t+1,\tau)=-\theta_1(t,\tau), \ \ \theta_1(t+\tau,\tau)=q^{-1/2}e^{-2\pi\mathbf{i}t}\theta_1(t,\tau),\nonumber
  \end{equation}
  \begin{equation}
    \theta_2(t+1,\tau)=\theta_2(t,\tau), \ \ \theta_2(t+\tau,\tau)=-q^{-1/2}e^{-2\pi\mathbf{i}t}\theta_2(t,\tau),\nonumber
  \end{equation}
  \begin{equation}
    \theta_3(t+1,\tau)=\theta_3(t,\tau), \ \ \theta_3(t+\tau,\tau)=-q^{-1/2}e^{-2\pi\mathbf{i}t}\theta_3(t,\tau),\nonumber
  \end{equation}
where $q=e^{\pi\mathbf{i}\tau}$.

Using the formulas in the first column, for $a\in 2\mathbb{Z}$, we easily verify that the $\theta,\theta_1,\theta_2,\theta_3$ terms in $\mathcal{L}'$ remain unchanged. So (1) holds.

For (2), under the replacement $t\to t+a\tau$ with $a$ an even integer, we apply the second column formulas to obtain:
\begin{equation}
  \theta(t+a\tau,\tau)=(-1)^ae^{-2\pi\mathbf{i}(at+a^2\tau/2)}\theta(t,\tau),\nonumber
\end{equation}
\begin{equation}
  \theta_1(t+a\tau,\tau)=e^{-2\pi\mathbf{i}(at+a^2\tau/2)}\theta_1(t,\tau),\nonumber
\end{equation}
\begin{equation}
  \theta_2(t+a\tau,\tau)=(-1)^ae^{-2\pi\mathbf{i}(at+a^2\tau/2)}\theta_2(t,\tau),\nonumber
\end{equation}
\begin{equation}
  \theta_3(t+a\tau,\tau)=e^{-2\pi\mathbf{i}(at+a^2\tau/2)}\theta_3(t,\tau).\nonumber
\end{equation}
Then, under the replacement $t\to t+a\tau$ for $a$ an even integer, the following transformation property holds:
\begin{equation}
  \theta_\mu(z^{h_\nu}_\nu+n_\nu(t+a\tau),\tau)=e^{-2\pi\mathbf{i}a(n_\nu z^{h_\nu}_\nu+n^2_\nu(t+a^2\tau/2))}\theta_\mu(z^{h_\nu}_\nu+n_\nu t,\tau),
\end{equation}
where $\theta_\mu\in\{\theta,\theta_1,\theta_2,\theta_3\}$.

The condition $3p_1(L)_{S^1}+p_1(V)_{S^1}=p_1(TM)_{S^1}$ implies that
$$3(u+\sigma\bar{t})^2+\sum_\nu\sum_{h_\nu}(z_\nu^{h_\nu}+n_\nu\bar{t})^2=\sum_{j}y_j^2+
\sum_\beta\sum_{i_\beta}(x^{i_\beta}_\beta+m_\beta\bar{t})^2,$$
where $\bar{t}\in H^2_{S^1}(pt)$ is a generator of $H^2_{S^1}(pt)$.
We obtain
$$3u^2+\sum_\nu\sum_{h_\nu}(z_\nu^{h_\nu})^2=\sum_{j}y_j^2+\sum_\beta\sum_{i_\beta}(x^{i_\beta}_\beta)^2,$$
$$3u\sigma+\sum_\nu\sum_{h_\nu}z_\nu^{h_\nu}n_\nu=\sum_\beta\sum_{i_\beta}x^{i_\beta}_\beta m_\beta,$$
$$3\sigma^2+\sum_\nu\sum_{h_\nu}(n_\nu)^2=\sum_\beta\sum_{i_\beta}(m_\beta)^2.$$

Thus, for each connected component $F_\alpha$ of the fixed locus of $M$, $\mathcal{L}'_\lambda$ remains unchanged. This completes the proof of item (2).
\end{proof}

This lemma shows that for fixed $\tau$, the functions $\mathcal{L}_\lambda'(t,\tau)$ are meromorphic on the torus $\mathbf{C}/2\mathbb{Z}\times2\mathbb{Z}\tau$. Therefore, to establish rigidity we only need to prove that they are holomorphic in $t$. We will in fact prove that they are holomorphic in $(t,\tau)$ on $\mathbf{C}\times\mathbf{H}$.

Given
$$g=\left(\begin{array}{cc}
\ a & b  \\
 c  & d
\end{array}\right)\in SL_2(\mathbb{Z})$$
define its modular transformation on $\mathbf{C}\times\mathbf{H}$ by
$$g(t,\tau)=\left(\frac{\tau}{c\tau+d},\frac{a\tau+b}{c\tau+d}\right).$$
This defines a group action. For the two generators of $SL_2(\mathbb{Z})$,
$$S=\left(\begin{array}{cc}
\ 0 & -1  \\
 1  & 0
\end{array}\right),~T=\left(\begin{array}{cc}
\ 1 &  1 \\
 0  & 1
\end{array}\right)$$
we have
$$S(t,\tau)=\left(\frac{t}{\tau},-\frac{1}{\tau}\right),~T(t,\tau)=(t,\tau+1).$$
Then we have the following transformation formulas.
\begin{lem}\label{le3.4}
  (1) Under the action of $S$ on $\mathcal{L}_\lambda'$ and assuming $3p_1(L)_{S^1}+p_1(V)_{S^1}=p_1(TM)_{S^1}$, 
  \begin{equation}\label{eq:3.20}
    \mathcal{L}_1'\left(\frac{t}{\tau},-\frac{1}{\tau}\right)=2^{\bar{l}-\bar{r}}\tau^{2k}
    \mathcal{L}_2'(t,\tau),
  \end{equation}
    \begin{equation}
     \mathcal{L}_2'\left(\frac{t}{\tau},-\frac{1}{\tau}\right)=2^{\bar{r}-\bar{l}}\tau^{2k}
    \mathcal{L}_1'(t,\tau),
  \end{equation}
  \begin{equation}\label{eq:3.22}
     \mathcal{L}_3'\left(\frac{t}{\tau},-\frac{1}{\tau}\right)=\tau^{2k}\mathcal{L}_3'(t,\tau).
  \end{equation}
  (2) Under the action of $T$ on $\mathcal{L}_\lambda'$, 
  \begin{equation}\label{eq:3.23}
    \mathcal{L}_1'(t,\tau+1)=\mathcal{L}_1'(t,\tau),
  \end{equation}
  \begin{equation}
    \mathcal{L}_2'(t,\tau+1)=\mathcal{L}_3'(t,\tau),
  \end{equation}
  \begin{equation}\label{eq:3.25}
    \mathcal{L}_3'(t,\tau+1)=\mathcal{L}_2'(t,\tau).
  \end{equation}
\end{lem}
\begin{proof}
  We have the following transformation laws of Jacobi theta-functions under the actions of $S$ and $T$ (cf. \cite{Ch}):
$$\theta\left(\frac{t}{\tau},-\frac{1}{\tau}\right)=\frac{1}{\mathbf{i}}\sqrt{\frac{\tau}{\mathbf{i}}}
e^{\pi\mathbf{i}t^2/\tau}\theta(t,\tau),~~\theta(t,\tau+1)=e^{\pi\mathbf{i}/4}\theta(t,\tau);$$
$$\theta_1\left(\frac{t}{\tau},-\frac{1}{\tau}\right)=\frac{1}{\mathbf{i}}\sqrt{\frac{\tau}{\mathbf{i}}}
e^{\pi\mathbf{i}t^2/\tau}\theta_2(t,\tau),~~\theta_1(t,\tau+1)=e^{\pi\mathbf{i}/4}\theta_1(t,\tau);$$
$$\theta_2\left(\frac{t}{\tau},-\frac{1}{\tau}\right)=\frac{1}{\mathbf{i}}\sqrt{\frac{\tau}{\mathbf{i}}}
e^{\pi\mathbf{i}t^2/\tau}\theta_1(t,\tau),~~\theta_2(t,\tau+1)=\theta_3(t,\tau);$$
$$\theta_3\left(\frac{t}{\tau},-\frac{1}{\tau}\right)=\frac{1}{\mathbf{i}}\sqrt{\frac{\tau}{\mathbf{i}}}
e^{\pi\mathbf{i}t^2/\tau}\theta_3(t,\tau),~~\theta_3(t,\tau+1)=e^{\pi\mathbf{i}/4}\theta_2(t,\tau);$$
$$\theta'\left(0,-\frac{1}{\tau}\right)=\left(\frac{\tau}{\mathbf{i}}\right)^{3/2}\theta'(0,\tau),
~~\theta'(0,\tau+1)=e^{\pi\mathbf{i}/4}\theta'(0,\tau).$$
Then we obtain the following transformation formulas:
\begin{equation}\label{eq:3.26}
  y_j\frac{\theta'(0,-1/\tau)}{\theta(y_j,-1/\tau)}=e^{-\pi\mathbf{i}\tau y^2_j}\tau y_j\frac{\theta'(0,\tau)}{\theta(\tau y_j,\tau)},
\end{equation}
\begin{equation}
  \frac{\theta'(0,-1/\tau)}{\theta(x_{\beta}^{i_\beta}+m_{\beta}t/\tau,-1/\tau)}=e^{-\pi\mathbf{i}\tau(x_{\beta}^{i_\beta}
  +m_{\beta}t/\tau)^2}\tau \frac{\theta'(0,\tau)}{\theta(\tau x_{\beta}^{i_\beta}+m_{\beta}t,\tau)},
\end{equation}
\begin{equation}
  \frac{\theta_1(u+\sigma t/\tau,-1/\tau)}{\theta_1(0,-1/\tau)}=e^{\pi\mathbf{i}\tau (u+\sigma t/\tau)^2}\frac{\theta_2(\tau u+\sigma t,\tau)}{\theta_2(0,\tau)},
\end{equation}
\begin{equation}
  \frac{\theta_2(u+\sigma t/\tau,-1/\tau)}{\theta_2(0,-1/\tau)}=e^{\pi\mathbf{i}\tau (u+\sigma t/\tau)^2}\frac{\theta_1(\tau u+\sigma t,\tau)}{\theta_1(0,\tau)},
\end{equation}
\begin{equation}
  \frac{\theta_3(u+\sigma t/\tau,-1/\tau,-1/\tau)}{\theta_3(0,-1/\tau)}=e^{\pi\mathbf{i}\tau (u+\sigma t/\tau)^2}\frac{\theta_3(\tau u+\sigma t,\tau)}{\theta_3(0,\tau)},
\end{equation}
\begin{equation}
  \frac{\theta_1(z^{h_\nu}_\nu+n_\nu t/\tau,-1/\tau)}{\theta_1(0,-1/\tau)}=e^{\pi\mathbf{i}\tau (z^{h_\nu}_\nu+n_\nu t/\tau)^{2}}\frac{\theta_2(\tau z^{h_\nu}_\nu+n_\nu t,\tau)}{\theta_2(0,\tau)}.
\end{equation}
\begin{equation}
  \frac{\theta_2(z^{h_\nu}_\nu+n_\nu t/\tau,-1/\tau)}{\theta_1(0,-1/\tau)}=e^{\pi\mathbf{i}\tau (z^{h_\nu}_\nu+n_\nu t/\tau)^{2}}\frac{\theta_1(\tau z^{h_\nu}_\nu+n_\nu t,\tau)}{\theta_1(0,\tau)}.
\end{equation}
\begin{equation}\label{eq:3.33}
  \frac{\theta_3(z^{h_\nu}_\nu+n_\nu t/\tau,-1/\tau)}{\theta_1(0,-1/\tau)}=e^{\pi\mathbf{i}\tau (z^{h_\nu}_\nu+n_\nu t/\tau)^{2}}\frac{\theta_3(\tau z^{h_\nu}_\nu+n_\nu t,\tau)}{\theta_3(0,\tau)}.
\end{equation}
Combining the condition $3p_1(L)_{S^1}+p_1(V)_{S^1}=p_1(M)_{S^1}$ with \eqref{eq:3.26}--\eqref{eq:3.33}, we obtain \eqref{eq:3.20}--\eqref{eq:3.22}.

For (2), using the transformation laws of Jacobi theta-functions under the action of $T$, we easily verify \eqref{eq:3.23}--\eqref{eq:3.25}.
\end{proof}

\begin{lem}
  For any function $\mathcal{L}_\lambda'(t,\tau)$, its modular transformation is holomorphic in $(t,\tau)\in\mathbf{R}\times\mathbf{H}$.
\end{lem}
\begin{proof}
The proof is essentially the same as that of [12, Lemma 2.3], except that we use Proposition 3.2 instead of the corresponding Lefschetz fixed point formulas.
\end{proof}

\textit{Proof of Theorem 3.1.} We prove that $\mathcal{L}_\lambda'(t,\tau)$ is holomorphic on $\mathbf{C}\times\mathbf{H}$, which implies the rigidity stated in Theorem 3.1. Denote by $L'$ one of the functions $\{\mathcal{L}_\lambda',~\lambda=1,2,3\}$.
 By \eqref{eq:3.10}--\eqref{eq:3.12}, the possible poles of $L'(t,\tau)$ can be written in the form $t=\frac{k}{l}(c\tau+d)$ for integers $k,l,c,d$ with $(c, d)=1$.

We can always find integers $a,b$ such that $ad-bc=1$. Then the matrix $g_0=\left(\begin{array}{cc}
\ d & -b  \\
 -c  & a
\end{array}\right)\in SL_2(\mathbb{Z})$ induces an action
$$L'(g_0(t,\tau))=L'\left(\frac{t}{-c\tau+a},\frac{d\tau-b}{-c\tau+a}\right).$$
Now, if $t=\frac{k}{l}(c\tau+d)$ is a polar divisor of $L'(t,\tau)$, then one polar divisor of $L'(g_0(t,\tau))$ is given by
$$\frac{t}{-c\tau+a}=\frac{k}{l}\left(c\frac{d\tau-b}{-c\tau+a}+d\right),$$
which yields $t=\frac{k}{l}$.

By Lemma 3.4, we know that, up to a constant, $L'(g_0(t,\tau))$ is still one of $\{\mathcal{L}_\lambda'(t,\tau)\}$. This contradicts Lemma 3.5, thus completing the proof of Theorem 3.1.

Let $\dim M=4k+2$ and define the virtual complex vector bundle $\Theta^*(T_{\mathbf{C}}M,L_{\bf{R}}\otimes\bf{C})$ over $M$ by
\begin{equation}
    \begin{split}
        \Theta^*(T_{\mathbf{C}}M,L_{\bf{R}}\otimes\mathbf{C})=\bigotimes _{n=1}^{\infty}S_{q^n}(\widetilde{T_{\mathbf{C}}M})\otimes
\bigotimes _{m=1}^{\infty}\wedge_{-q^m}(\widetilde{L_{\bf{R}}\otimes\mathbf{C}}).\nonumber
    \end{split}
\end{equation}

We define the generalized Witten forms
\begin{equation}
    W^*_1(M,L,V)=\left\{\widehat{A}(TM){\rm exp}(\frac{c}{2}){\rm ch}(\Theta^*(T_{\mathbf{C}}M,L_{\bf{R}}\otimes\mathbf{C})){\rm ch}(Q_1(V_{\mathbf{C}}))\right\}^{4k+2},
\end{equation}
\begin{equation}
    W^*_2(M,L,V)=\left\{\widehat{A}(TM){\rm exp}(\frac{c}{2}){\rm ch}(\Theta^*(T_{\mathbf{C}}M,L_{\bf{R}}\otimes\mathbf{C})){\rm ch}(Q_2(V_{\mathbf{C}}))\right\}^{4k+2},
\end{equation}
\begin{equation}
    W^*_3(M,L,V)=\left\{\widehat{A}(TM){\rm exp}(\frac{c}{2}){\rm ch}(\Theta^*(T_{\mathbf{C}}M,L_{\bf{R}}\otimes\mathbf{C})){\rm ch}(Q_3(V_{\mathbf{C}}))\right\}^{4k+2}.
\end{equation}

Define the generalized Witten genus
\begin{equation}
\mathcal{W}^*_\lambda(M,L,V):=\int_{M^{4k}}W^*_\lambda(M,L,V), \ \  \lambda=1,2,3.
\end{equation}

Now assume that $M$ admits an $S^1$-action that lifts to an action on $L$. Moreover, assume that this action preserves the given spin$^c$-structure associated to $(M, L)$, as well as the metrics and connections involved. Let $V$ be equipped with an $S^1$-action that restricts on each $V$-fiber over $M$ to a linear action preserving that fiber. Then the associated twisted Dirac operator is $S^1$-equivariant. We then have the following rigidity theorems.

\begin{thm}
  For a $(4k+2)$-dimensional connected spin$^c$ manifold with a non-trivial $S^1$-action, if $p_1(L)_{S^1}+p_1(V)_{S^1}=p_1(TM)_{S^1}$, then the twisted Dirac operators $$\mathcal{D}_c\otimes\Theta^*(T_{\mathbf{C}}M,L_{\bf{R}}\otimes\mathbf{C})\otimes Q_\lambda(V_{\mathbf{C}}),~\lambda=1,2,3$$ are rigid.
\end{thm}

To prove this rigidity theorem, we first calculate the corresponding Lefschetz numbers.
\begin{prop}\label{pro3.7}
 The Lefschetz numbers of $\mathcal{D}_c\otimes\Theta^*(T_{\mathbf{C}}M,L_{\bf{R}}\otimes\mathbf{C})\otimes Q_\lambda(V_{\mathbf{C}}),~\lambda=1,2,3$ are
\begin{equation}\label{eq:3.38}
  \begin{split}
  \mathcal{L}^*_1(g;\tau)=&2^{\bar{l}-\bar{r}}\left(\frac{-\mathbf{i}}{\pi}\right)^{\bar{r}}
  \sum_{\alpha}\int_{F_\alpha}\left(\prod^{s}_{j=1}
  y_j\frac{\theta'(0,\tau)}{\theta(y_j,\tau)}\prod^{r}_{\beta=1}
  \prod^{\dim\mathbf{N}_\beta}_{i_\beta=1}\frac{\theta'(0,\tau)}{\theta(x_{\beta}^{i_\beta}+
  m_{\beta}t,\tau)}\right)\\
  &\cdot\frac{\sqrt{-1}\theta(u+\sigma t,\tau)}{\theta_1(0,\tau)\theta_2(0,\tau)
  \theta_3(0,\tau)}\cdot\prod_{\nu=0}^l\prod_{h_\nu=1}^{\dim V_\nu}\frac{\theta_1(z_\nu^{h_\nu}+n_\nu t,\tau)}{\theta_1(0,\tau)},
  \end{split}
\end{equation}
\begin{equation}
  \begin{split}
  \mathcal{L}^*_2(g;\tau)=&-\left(\frac{\mathbf{i}}{2\pi}\right)^{\bar{r}}
  \sum_{\alpha}\int_{F_\alpha}\left(\prod^{s}_{j=1}
  y_j\frac{\theta'(0,\tau)}{\theta(y_j,\tau)}\prod^{r}_{\beta=1}
  \prod^{\dim\mathbf{N}_\beta}_{i_\beta=1}\frac{\theta'(0,\tau)}{\theta(x_{\beta}^{i_\beta}+
  m_{\beta}t,\tau)}\right)\\
  &\cdot\frac{\sqrt{-1}\theta(u+\sigma t,\tau)}{\theta_1(0,\tau)\theta_2(0,\tau)
  \theta_3(0,\tau)}\cdot\prod_{\nu=0}^l\prod_{h_\nu=1}^{\dim V_\nu}\frac{\theta_2(z_\nu^{h_\nu}+n_\nu t,\tau)}{\theta_2(0,\tau)},
  \end{split}
\end{equation}
\begin{equation}\label{eq:3.40}
  \begin{split}
  \mathcal{L}^*_3(g;\tau)=&-\left(\frac{\mathbf{i}}{2\pi}\right)^{\bar{r}}
  \sum_{\alpha}\int_{F_\alpha}\left(\prod^{s}_{j=1}
  y_j\frac{\theta'(0,\tau)}{\theta(y_j,\tau)}\prod^{r}_{\beta=1}
  \prod^{\dim\mathbf{N}_\beta}_{i_\beta=1}\frac{\theta'(0,\tau)}{\theta(x_{\beta}^{i_\beta}+
  m_{\beta}t,\tau)}\right)\\
  &\cdot\frac{\sqrt{-1}\theta(u+\sigma t,\tau)}{\theta_1(0,\tau)\theta_2(0,\tau)
  \theta_3(0,\tau)}\cdot\prod_{\nu=0}^l\prod_{h_\nu=1}^{\dim V_\nu}\frac{\theta_3(z_\nu^{h_\nu}+n_\nu t,\tau)}{\theta_3(0,\tau)},
  \end{split}
\end{equation}
\end{prop}
\begin{proof}
 By the Lefschetz fixed point formula (3.13), we calculate directly:
\begin{equation}\label{eq:3.41}
 \begin{split}
  \mathcal{L}^*_\lambda(g,\tau)=&\prod_{j=1}^{s}\frac{\pi y_j}{\sin \pi y_j}\prod^{r}_{\beta=1}\prod^{\dim\mathbf{N}_\beta}_{i_\beta=1}\frac{1}{2\mathbf{i}\sin(\pi
  (x_{\beta}^{i_\beta}+m_{\beta}t))}e^{\pi\sqrt{-1}(u+\sigma t)}\\
  &\cdot{\rm ch}_g(\Theta^*(T_{\mathbf{C}}M,L_{\bf{R}}\otimes\mathbf{C})){\rm ch}_g(Q_\lambda(V_{\mathbf{C}})),
  \end{split}
\end{equation}
\begin{equation}
 \begin{split}
  e^{\pi\sqrt{-1}(u+\sigma t)}
  &{\rm ch}_g(\Theta^*(T_{\mathbf{C}}M,L_{\bf{R}}\otimes\mathbf{C}))\\
  =&\prod_{j=1}^{s}\frac{\sin(\pi y_j)}{\pi}\left(\frac{\theta'(0,\tau)}{\theta(y_j,\tau)}\right)\prod^{r}_{\beta=1}
  \prod^{\dim\mathbf{N}_\beta}_{i_\beta=1}\frac{\sin(\pi
  (x_{\beta}^{i_\beta}+m_{\beta}t))}{\pi}\\
  &\cdot\left(\frac{\theta'(0,\tau)}{\theta(x_{\beta}^{i_\beta}+
  m_{\beta}t,\tau)}\right)
  \cdot\frac{\sqrt{-1}\theta(u+\sigma t,\tau)}{\theta_1(0,\tau)\theta_2(0,\tau)
  \theta_3(0,\tau)},
 \end{split}
\end{equation}
\begin{equation}
  {\rm ch}_g\left(\Delta(V)\otimes
  \bigotimes_{n=1}^{\infty}\Lambda_{q^n}(\widetilde{V_{\mathbf{C}}})\right)=2^{\bar{l}}
  \prod_{\nu=0}^l\prod_{h_\nu=1}^{\dim V_\nu}\frac{\theta_1(z_\nu^{h_\nu}+n_\nu t,\tau)}{\theta_1(0,\tau)},
\end{equation}
\begin{equation}
 {\rm ch}_g\left(
\bigotimes_{n=1}^{\infty}\Lambda_{-q^{n-\frac{1}{2}}}(\widetilde{V_{\mathbf{C}}})\right) =\prod_{\nu=0}^l\prod_{h_\nu=1}^{\dim V_\nu}\frac{\theta_2(z_\nu^{h_\nu}+n_\nu t,\tau)}{\theta_2(0,\tau)},
\end{equation}
\begin{equation}\label{eq:3.46}
  {\rm ch}_g\left(
  \bigotimes_{n=1}^{\infty}\Lambda_{q^{n-\frac{1}{2}}}(\widetilde{V_{\mathbf{C}}})\right)
  =\prod_{\nu=0}^l\prod_{h_\nu=1}^{\dim V_\nu}\frac{\theta_3(z_\nu^{h_\nu}+n_\nu t,\tau)}{\theta_3(0,\tau)}.
\end{equation}
Thus, Proposition \ref{pro3.7} follows from \eqref{eq:3.41}--\eqref{eq:3.46}.
\end{proof}

In what follows, we view the above expressions as defining functions $\mathcal{L}^*_\lambda(t,\tau)$ such that $\mathcal{L}^*_\lambda(t,\tau)=\mathcal{L}^*_\lambda(g;\tau)$.
The expressions for $\mathcal{L}^*_\lambda(t,\tau)$ involve only complex-differentiable functions, so we can extend their domain to every complex number $t$ and choice of $\tau$ in the open upper half-plane, i.e., $\mathcal{L}^*_\lambda(t,\tau)$ where the expression exists in $\mathbf{C}\times\mathbf{H}$. The Witten rigidity theorems are equivalent to the statement that these $\mathcal{L}^*_\lambda(t,\tau)$ are independent of $t$.  We then have the following lemma.

\begin{lem}
Let $(t,\tau)\in\mathbf{C}\times\mathbf{H}$ be in the domain of $\mathcal{L}^*_\lambda(t,\tau)$.\\
(1) Then $\mathcal{L}^*_\lambda(t,\tau)=\mathcal{L}^*_\lambda(t+a,\tau)$ for any $a\in 2\mathbb{Z}$.\\
(2) If $p_1(L)_{S^1}+p_1(V)_{S^1}=p_1(TM)_{S^1}$, then $\mathcal{L}^*_\lambda(t,\tau)=\mathcal{L}^*_\lambda(t+a\tau,\tau)$ for any $a\in 2\mathbb{Z}$.
\end{lem}
\begin{proof}
Similar to Lemma \ref{le3.3}.
\end{proof}

Next, we prove that these $\mathcal{L}^*_\lambda$ are holomorphic in $(t,\tau)$ on $\mathbf{C}\times\mathbf{H}$. Similar to Lemma \ref{le3.4}, we have the following lemma.
\begin{lem}
  (1) Under the action of $S$ on $\mathcal{L}^*_\lambda$ and assuming $3p_1(L)_{S^1}+p_1(V)_{S^1}=p_1(TM)_{S^1}$, 
  \begin{equation}
    \mathcal{L}_1^*\left(\frac{t}{\tau},-\frac{1}{\tau}\right)=2^{\bar{l}-\bar{r}}\tau^{s+\bar{l}-1}
    \mathcal{L}_2^*(t,\tau),
  \end{equation}
    \begin{equation}
     \mathcal{L}_2^*\left(\frac{t}{\tau},-\frac{1}{\tau}\right)=2^{\bar{r}-\bar{l}}\tau^{s+\bar{l}-1}
    \mathcal{L}_1^*(t,\tau),
  \end{equation}
  \begin{equation}
     \mathcal{L}_3^*\left(\frac{t}{\tau},-\frac{1}{\tau}\right)=\tau^{s+\bar{l}-1}\mathcal{L}_3^*(t,\tau).
  \end{equation}
  (2) Under the action of $T$ on $\mathcal{L}^*_\lambda$, 
  \begin{equation}
    \mathcal{L}_1^*(t,\tau+1)=\mathcal{L}_1^*(t,\tau),
  \end{equation}
  \begin{equation}
    \mathcal{L}_2^*(t,\tau+1)=\mathcal{L}_3^*(t,\tau),
  \end{equation}
  \begin{equation}
    \mathcal{L}_3^*(t,\tau+1)=\mathcal{L}_2^*(t,\tau).
  \end{equation}
\end{lem}

\begin{lem}
  For any function $\mathcal{L}^*_\lambda, \ \ \lambda=\{1,2,3\}$, its modular transformation is holomorphic in $(t,\tau)\in\mathbf{R}\times\mathbf{H}$.
\end{lem}
\begin{proof}
Similar to Lemma 3.5.
\end{proof}

\textit{Proof of Theorem 3.6.} We prove that these $\mathcal{L}^*_\lambda$ are holomorphic on $\mathbf{C}\times\mathbf{H}$, which implies the rigidity of Theorem 3.6. Denote by $L^*$ one of the functions $\{\mathcal{L}^*_1,\mathcal{L}^*_2,\mathcal{L}^*_3\}$. By \eqref{eq:3.38}--\eqref{eq:3.40}, the possible poles of $L^*(t,\tau)$ can be written in the form $t=\frac{k}{l}(c\tau+d)$ for integers $k,l,c,d$ with $(c, d)=1$.

We can always find integers $a,b$ such that $ad-bc=1$. Then the matrix $g_1=\left(\begin{array}{cc}
\ d & -b  \\
 -c  & a
\end{array}\right)\in SL_2(\mathbb{Z})$ induces an action
$$L^*(g_1(t,\tau))=L^*\left(\frac{t}{-c\tau+a},\frac{d\tau-b}{-c\tau+a}\right).$$
Now, if $t=\frac{k}{l}(c\tau+d)$ is a polar divisor of $L^*(t,\tau)$, then one polar divisor of $L^*(g_1(t,\tau))$ is given by
$$\frac{t}{-c\tau+a}=\frac{k}{l}\left(c\frac{d\tau-b}{-c\tau+a}+d\right),$$
which yields $t=\frac{k}{l}$.

By Lemma 3.9, we know that, up to a constant, $L^*(g_1(t,\tau))$ is still one of $\{\mathcal{L}^*_1,\mathcal{L}^*_2,\mathcal{L}^*_3\}$. This contradicts Lemma 3.10, thus completing the proof of Theorem 3.6.

Let $V$ be an oriented real rank-$2\bar{l}$ vector bundle on a manifold $M$, equipped with an $S^1$-action that restricts on each $V$-fiber over $M$ to a linear action preserving that fiber. Set
\begin{equation}
 Q(V_{\mathbf{C}})=Q_1(V_{\mathbf{C}})\otimes Q_2(V_{\mathbf{C}})\otimes Q_3(V_{\mathbf{C}}).
\end{equation}

Define the generalized Witten forms
\begin{equation}
    W(M,L,V)=\left\{\widehat{A}(TM){\rm exp}(\frac{c}{2}){\rm ch}(\Theta(T_{\mathbf{C}}M,L_{\bf{R}}\otimes\mathbf{C})){\rm ch}(Q(V_{\mathbf{C}}))\right\}^{4k}
\end{equation}
if $\dim M=4k$ and
\begin{equation}
    W^*(M,L,V)=\left\{\widehat{A}(TM){\rm exp}(\frac{c}{2}){\rm ch}(\Theta^*(T_{\mathbf{C}}M,L_{\bf{R}}\otimes\mathbf{C})){\rm ch}(Q(V_{\mathbf{C}}))\right\}^{4k+2}
\end{equation}
if $\dim M=4k+2$. We can also define the corresponding generalized Witten genus
\begin{equation}
\mathcal{W}(M,L,V):=\int_{M^{4k}}W(M,L,V)
\end{equation}
if $\dim M=4k$ and
\begin{equation}
\mathcal{W}^*(M,L,V):=\int_{M^{4k+2}}W^*(M,L,V)
\end{equation}
if $\dim M=4k+2$. Then the associated twisted Dirac operator is $S^1$-equivariant. Using Liu's method \cite{Li2}, we obtain the following theorem in the even-dimensional case.
\begin{thm}
 (i) For a $4k$-dimensional connected spin$^c$ manifold with a non-trivial $S^1$-action, if $3p_1(L)_{S^1}+3p_1(V)_{S^1}=p_1(TM)_{S^1}$, then the twisted Dirac operators $$\mathcal{D}_c\otimes\Theta(T_{\mathbf{C}}M,L_{\bf{R}}\otimes\mathbf{C})\otimes Q(V_{\mathbf{C}})$$ are rigid.
 
 (ii) For a $(4k+2)$-dimensional connected spin$^c$ manifold with a non-trivial $S^1$-action, if $p_1(L)_{S^1}+3p_1(V)_{S^1}=p_1(TM)_{S^1}$, then the twisted Dirac operators $$\mathcal{D}_c\otimes\Theta^*(T_{\mathbf{C}}M,L_{\bf{R}}\otimes\mathbf{C})\otimes Q(V_{\mathbf{C}})$$ are rigid.
\end{thm}

\section{Twisted Toeplitz operator and Witten rigidity theorem in odd dimensions}
We recall the odd Chern character of a smooth map $g$ from $M$ to the general linear group $GL(N,\mathbf{C})$
with $N$ a positive integer (see \cite{HY}). Let $d$ denote a trivial connection on $\mathbf{C}^{N}|_{M}$. We denote by $c_n(M,[g])$ the cohomology class associated to the closed $n$-form
\begin{equation}
  c_n(\mathbf{C}^{N}|_M,g,d)=\left(\frac{1}{2\pi\sqrt{-1}}\right)^{\frac{(n+1)}{2}}\mathrm{Tr}[(g^{-1}dg)^n].
\end{equation}
The odd Chern character form ${\rm ch}(\mathbf{C}^{N}|_M,g,d)$ associated to $g$ and $d$ is by definition
\begin{equation}
  {\rm ch}(\mathbf{C}^{N}|_M,g,d)=\sum^{\infty}_{n=1}\frac{n!}{(2n+1)!}c_{2n+1}((\mathbf{C}^{N}|_M,g,d)).
\end{equation}
Let the connection $\nabla_{\bar{u}}$ on the trivial bundle $\mathbf{C}^{N}|_M$ be defined by
\begin{equation}
  \nabla_{\bar{u}}=(1-\bar{u})d+\bar{u}g^{-1}\cdot d \cdot g,\ \ \bar{u}\in[0,1].
\end{equation}
Then we have
\begin{equation}
  d{\rm ch}(\mathbf{C}^{N}|_M,g,d)={\rm ch}(\mathbf{C}^{N}|_M,d)-{\rm ch}(\mathbf{C}^{N}|_M,g^{-1}\cdot d\cdot g).
\end{equation}

Now let $g:M\to SO(N)$ and assume that $N$ is even and large enough. Let $E$ denote the trivial real vector bundle of rank $N$ over $M$. We equip $E$ with the canonical trivial metric and trivial connection $d$. Set
$$\nabla_{\bar{u}}=d+\bar{u}g^{-1}dg,\ \ \bar{u}\in[0,1].$$
Let $R_{\bar{u}}$ be the curvature of $\nabla_{\bar{u}}$, then
\begin{equation}
  R_{\bar{u}}=(\bar{u}^2-\bar{u})(g^{-1}dg)^2.
\end{equation}
We also consider the complexification of $E$, and $g$ extends to a unitary automorphism of $E_{\mathbf{C}}$. The connection $\nabla_{\bar{u}}$ extends to a Hermitian connection on $E_{\mathbf{C}}$ with curvature still given by (3.6). Let $\Delta(E)$ be the spinor bundle of $E$, which is a trivial Hermitian
bundle of rank $2^{\frac{N}{2}}$. We assume that $g$ lifts to the Spin group ${\rm Spin}(N):g^{\Delta}:M\to {\rm Spin}(N)$. So $g^{\Delta}$ can be viewed as an automorphism of $\Delta(E)$ preserving the Hermitian metric. We lift $d$ on $E$ to a trivial Hermitian connection $d^{\Delta}$ on $\Delta(E)$, then
\begin{equation}
  \nabla_{\bar{u}}^{\Delta}=(1-\bar{u})d^{\Delta}+\bar{u}(g^{\Delta})^{-1}\cdot d^{\Delta} \cdot g^{\Delta},\ \ \bar{u}\in[0,1]
\end{equation}
lifts $\nabla_{\bar{u}}$ on $E$ to $\Delta(E)$. Let $Q_j(E),j=1,2,3$ be the virtual bundles defined as follows:
\begin{equation}
Q_1(E)=\triangle(E)\otimes
   \bigotimes _{n=1}^{\infty}\wedge_{q^n}(\widetilde{E_\mathbf{C}});
\end{equation}
\begin{equation}
Q_2(E)=\bigotimes _{n=1}^{\infty}\wedge_{-q^{n-\frac{1}{2}}}(\widetilde{E_\mathbf{C}});
\end{equation}
\begin{equation}
Q_3(E)=\bigotimes _{n=1}^{\infty}\wedge_{q^{n-\frac{1}{2}}}(\widetilde{E_\mathbf{C}}).
\end{equation}
Let $g$ on $E$ have a lift $g^{Q(E)}$ on $Q(E)$ and $\nabla_{\bar{u}}$ have a lift $\nabla^{Q(E)}_{\bar{u}}$ on $Q(E)$. Following \cite{HY}, we define ${\rm ch}(Q_j(E),g^{Q_j(E)},d,\tau),~j=1,2,3$ as 
\begin{equation}
{\rm ch}(Q_j(E),\nabla^{Q_j(E)}_0,\tau)-{\rm ch}(Q_j(E),\nabla^{Q_j(E)}_1,\tau)=d{\rm ch}(Q_j(E),g^{Q_j(E)},d,\tau),
\end{equation}
where
\begin{equation}
{\rm ch}(Q_1(E),g^{Q_1(E)},d,\tau)=-\frac{2^{N/2}}{8\pi^2}\int^1_0{\rm Tr}\left[g^{-1}dg\frac{\theta'_1(R_{\bar{u}}/(4\pi^2),\tau)}{\theta_1(R_{\bar{u}}/(4\pi^2),\tau)}\right]du,
\end{equation}
and for $j=2,3$
\begin{equation}
{\rm ch}(Q_j(E),g^{Q_j(E)},d,\tau)=-\frac{1}{8\pi^2}\int^1_0{\rm Tr}\left[g^{-1}dg\frac{\theta'_j(R_{\bar{u}}/(4\pi^2),\tau)}{\theta_j(R_{\bar{u}}/(4\pi^2),\tau)}\right]du.
\end{equation}

Let $M$ be a $(4k-1)$-dimensional spin$^c$ manifold. Set 
\begin{equation}
\begin{split}
    \overline{W}_1&(M,L,V)\\
    &=\left\{\widehat{A}(TM){\rm exp}(\frac{c}{2}){\rm ch}(\Theta(T_{\mathbf{C}}M,L_{\bf{R}}\otimes\mathbf{C})){\rm ch}(Q_1(V_{\mathbf{C}})){\rm ch}(Q_1(E),g^{Q_1(E)},d,\tau)\right\}^{4k-1},
\end{split}
\end{equation}
\begin{equation}
\begin{split}
    \overline{W}_2&(M,L,V)\\
    &=\left\{\widehat{A}(TM){\rm exp}(\frac{c}{2}){\rm ch}(\Theta(T_{\mathbf{C}}M,L_{\bf{R}}\otimes\mathbf{C})){\rm ch}(Q_2(V_{\mathbf{C}})){\rm ch}(Q_2(E),g^{Q_2(E)},d,\tau)\right\}^{4k-1},
\end{split}
\end{equation}
\begin{equation}
\begin{split}
    \overline{W}_3&(M,L,V)\\
    &=\left\{\widehat{A}(TM){\rm exp}(\frac{c}{2}){\rm ch}(\Theta(T_{\mathbf{C}}M,L_{\bf{R}}\otimes\mathbf{C})){\rm ch}(Q_3(V_{\mathbf{C}})){\rm ch}(Q_3(E),g^{Q_3(E)},d,\tau)\right\}^{4k-1}.
\end{split}
\end{equation}

Let $M$ be a closed smooth spin$^c$ Riemannian manifold which admits a circle action and assume that $M$ admits an $S^1$-action that lifts to an action on $L$. Let $V$ be an $S^1$-equivariant complex vector bundle over $M$ carrying an $S^1$-invariant Hermitian connection. In addition, we assume $g:M\to GL(N,\mathbf{C})$ is $S^1$-invariant, i.e.,
\begin{equation}
  g(hx)=g(x), \ {\rm for \ any} \ h\in S^1 \ {\rm and} \ x\in M.
\end{equation}
Thus the twisted Toeplitz operators $$\mathcal{T}_c\otimes\Theta(T_{\mathbf{C}}M,L_{\bf{R}}\otimes\mathbf{C})\otimes Q_\lambda(V_{\mathbf{C}})\otimes(Q_\lambda(E),g^{Q_\lambda(E)}),~\lambda=1,2,3$$ are $S^1$-equivariant.

We consider the fixed point case. Similarly, we have the following $S^1$-equivariant decomposition when restricted to $F_{\alpha}$:
\begin{equation}
  TM|_{F_{\alpha}}=\mathbf{N}_1\oplus\cdots\oplus \mathbf{N}_{r}\oplus TF_{\alpha}
\end{equation}
and
\begin{equation}
  V|_{F_\alpha}=V_1\oplus\cdots\oplus V_l\oplus V^{\mathbb{R}}_0,
\end{equation}
where $\mathbf{N}_{\beta}$ and $V_\nu$ are complex vector bundles, and $h$ acts on $\mathbf{N}_{\beta}$ and $V_\nu$ by $e^{2\pi \mathbf{i}m_{\beta}t}$ and $e^{2\pi \mathbf{i}n_{\nu}t}$, respectively. They have corresponding Chern roots as before. We then have the following rigidity theorems.
\begin{thm}
  For a $(4k-1)$-dimensional connected spin manifold with a non-trivial $S^1$-action, and assuming $g_*:\pi_1(M)\to\pi_1(SO(N))=\mathbb{Z}_2$ is trivial, if $$3p_1(L)_{S^1}+p_1(V)_{S^1}=p_1(TM)_{S^1}$$ and $c_3(E,g,d)=0$, then the Toeplitz operators $$\mathcal{T}_c\otimes\Theta(T_{\mathbf{C}}M,L_{\bf{R}}\otimes\mathbf{C})\otimes Q_\lambda(V_{\mathbf{C}})\otimes(Q_\lambda(E),g^{Q_\lambda(E)}),~\lambda=1,2,3$$ are rigid.
\end{thm}

First, we calculate the corresponding Lefschetz numbers.
\begin{prop}
 The Lefschetz numbers of $$\mathcal{T}_c\otimes\Theta(T_{\mathbf{C}}M,L_{\bf{R}}\otimes\mathbf{C})\otimes Q_\lambda(V_{\mathbf{C}})\otimes(Q_\lambda(E),g^{Q_\lambda(E)})$$ are
\begin{equation}\label{eq:4.19}
  \begin{split}
  \overline{\mathcal{L}}_1(g;\tau)&=2^{\bar{l}-\bar{r}}\left(\frac{-\mathbf{i}}{\pi}\right)^{\bar{r}}
  \sum_{\alpha}\int_{F_\alpha}\left(\prod^{s}_{j=1}
  y_j\frac{\theta'(0,\tau)}{\theta(y_j,\tau)}\prod^{r}_{\beta=1}
  \prod^{\dim\mathbf{N}_\beta}_{i_\beta=1}\frac{\theta'(0,\tau)}{\theta(x_{\beta}^{i_\beta}+
  m_{\beta}t,\tau)}\right)\\
  &\cdot\frac{\theta_1(u+\sigma t,\tau)\theta_2(u+\sigma t,\tau)\theta_3(u+\sigma t,\tau)}{\theta_1(0,\tau)\theta_2(0,\tau)
  \theta_3(0,\tau)}\cdot\prod_{\nu=0}^l\prod_{h_\nu=1}^{\dim V_\nu}\frac{\theta_1(z_\nu^{h_\nu}+n_\nu t,\tau)}{\theta_1(0,\tau)}\\
  &\cdot{\rm ch}(Q_1(E),g^{Q_1(E)},d,\tau),
  \end{split}
\end{equation}
\begin{equation}\label{eq:4.20}
  \begin{split}
  \overline{\mathcal{L}}_2(g;\tau)&=-\left(\frac{\mathbf{i}}{2\pi}\right)^{\bar{r}}
  \sum_{\alpha}\int_{F_\alpha}\left(\prod^{s}_{j=1}
  y_j\frac{\theta'(0,\tau)}{\theta(y_j,\tau)}\prod^{r}_{\beta=1}
  \prod^{\dim\mathbf{N}_\beta}_{i_\beta=1}\frac{\theta'(0,\tau)}{\theta(x_{\beta}^{i_\beta}+
  m_{\beta}t,\tau)}\right)\\
  &\cdot\frac{\theta_1(u+\sigma t,\tau)\theta_2(u+\sigma t,\tau)\theta_3(u+\sigma t,\tau)}{\theta_1(0,\tau)\theta_2(0,\tau)
  \theta_3(0,\tau)}\cdot\prod_{\nu=0}^l\prod_{h_\nu=1}^{\dim V_\nu}\frac{\theta_2(z_\nu^{h_\nu}+n_\nu t,\tau)}{\theta_2(0,\tau)}\\
  &\cdot{\rm ch}(Q_2(E),g^{Q_2(E)},d,\tau),
  \end{split}
\end{equation}
and
\begin{equation}\label{eq:4.21}
  \begin{split}
   \overline{\mathcal{L}}_3(g;\tau)&=-\left(\frac{\mathbf{i}}{2\pi}\right)^{\bar{r}}
  \sum_{\alpha}\int_{F_\alpha}\left(\prod^{s}_{j=1}
  y_j\frac{\theta'(0,\tau)}{\theta(y_j,\tau)}\prod^{r}_{\beta=1}
  \prod^{\dim\mathbf{N}_\beta}_{i_\beta=1}\frac{\theta'(0,\tau)}{\theta(x_{\beta}^{i_\beta}+
  m_{\beta}t,\tau)}\right)\\
  &\cdot\frac{\theta_1(u+\sigma t,\tau)\theta_2(u+\sigma t,\tau)\theta_3(u+\sigma t,\tau)}{\theta_1(0,\tau)\theta_2(0,\tau)
  \theta_3(0,\tau)}\cdot\prod_{\nu=0}^l\prod_{h_\nu=1}^{\dim V_\nu}\frac{\theta_3(z_\nu^{h_\nu}+n_\nu t,\tau)}{\theta_3(0,\tau)}\\
  &\cdot{\rm ch}(Q_3(E),g^{Q_3(E)},d,\tau).
  \end{split}
\end{equation}
\end{prop}
\begin{proof}
Similar to Proposition 3.2, we can calculate equations (4.19), (4.20) and (4.21) by using the Lefschetz fixed point formula (cf. \cite{Wa}), (4.11) and (4.12).
\end{proof}

Next, we view the above expressions as defining functions $\mathcal{F}_\lambda(t,\tau)$ for each $\lambda\in\{1,2,3\}$, i.e., $\mathcal{F}_\lambda(t,\tau)=\overline{\mathcal{L}}_\lambda(g;\tau)$.
The expressions for $\mathcal{F}_\lambda(t,\tau)$ involve only complex-differentiable functions, so we can extend their domain to every complex number $t$ and choice of $\tau$ in the open upper half-plane, i.e., $\mathcal{F}_\lambda(t,\tau)$ where the expression exists in $\mathbf{C}\times\mathbf{H}$. The Witten rigidity theorems are equivalent to the statement that these $\mathcal{F}_\lambda(t,\tau)$ are independent of $t$.  We then have the following lemma.
\begin{lem}
Let $(t,\tau)\in\mathbf{C}\times\mathbf{H}$ be in the domain of $\mathcal{F}_\lambda(t,\tau)$,~$\lambda\in\{1,2,3\}$.\\
(1) Then $\mathcal{F}_\lambda(t,\tau)=\mathcal{F}_\lambda(t+a,\tau)$ for any $a\in 2\mathbb{Z}$.\\
(2) If $3p_1(L)_{S^1}+p_1(V)_{S^1}=p_1(TM)_{S^1}$, then $\mathcal{F}_\lambda(t,\tau)=\mathcal{F}_\lambda(t+a\tau,\tau)$ for any $a\in 2\mathbb{Z}$.
\end{lem}
\begin{proof}
Similar to Lemma 3.3.
\end{proof}

Next, we prove that these $\mathcal{F}_\lambda$ are holomorphic in $(t,\tau)$ on $\mathbf{C}\times\mathbf{H}$. Similar to Lemma 3.4, we have the following lemma.
\begin{lem}
  (1) Under the assumptions $3p_1(L)_{S^1}+p_1(V)_{S^1}=p_1(TM)_{S^1}$ and $c_3(E,g,d)=0$, the action of $S$ on $\mathcal{F}_\lambda$ yields
  \begin{equation}\label{eq:4.22}
    \mathcal{F}_1\left(\frac{t}{\tau},-\frac{1}{\tau}\right)=2^{\bar{l}-\bar{r}+N/2}\tau^{2k}\mathcal{F}_2(t,\tau),
  \end{equation}
  \begin{equation}\label{eq:4.23}
    \mathcal{F}_2\left(\frac{t}{\tau},-\frac{1}{\tau}\right)=2^{-(\bar{l}-\bar{r}+N/2)}\tau^{2k}\mathcal{F}_1(t,\tau),
  \end{equation}
  \begin{equation}\label{eq:4.24}
    \mathcal{F}_3\left(\frac{t}{\tau},-\frac{1}{\tau}\right)=\tau^{2k}\mathcal{F}_3(t,\tau).
  \end{equation}
  (2) Under the action of $T$ on $\mathcal{F}_\lambda$,
  \begin{equation}\label{eq:4.25}
    \mathcal{F}_1(t,\tau+1)=\mathcal{F}_1(t,\tau),
  \end{equation}
   \begin{equation}\label{eq:4.26}
    \mathcal{F}_2(t,\tau+1)=\mathcal{F}_3(t,\tau),
  \end{equation}
   \begin{equation}\label{eq:4.27}
    \mathcal{F}_3(t,\tau+1)=\mathcal{F}_2(t,\tau).
  \end{equation}
\end{lem}
\begin{proof}
By Proposition 2.2 in \cite{HY}, if $c_3(E,g,d)=0$, then for any integer $i\geq 1$,
\begin{equation}
  \left\{{\rm ch}\left(Q_1(E),g^{Q_1(E)},d,-\frac{1}{\tau}\right)\right\}^{4i-1}=2^{N/2}\left\{\tau^{2i}{\rm ch}\left(Q_2(E),g^{Q_2(E)},d,\tau\right)\right\}^{4i-1},
\end{equation}
\begin{equation}
  \left\{{\rm ch}\left(Q_2(E),g^{Q_2(E)},d,-\frac{1}{\tau}\right)\right\}^{4i-1}=2^{-N/2}\left\{\tau^{2i}{\rm ch}\left(Q_1(E),g^{Q_1(E)},d,\tau\right)\right\}^{4i-1},
\end{equation}
\begin{equation}
  \left\{{\rm ch}\left(Q_3(E),g^{Q_3(E)},d,-\frac{1}{\tau}\right)\right\}^{4i-1}=\left\{\tau^{2i}{\rm ch}\left(Q_3(E),g^{Q_3(E)},d,\tau\right)\right\}^{4i-1}.
\end{equation}
Combining the condition $3p_1(L)_{S^1}+p_1(V)_{S^1}=p_1(TM)_{S^1}$ with \eqref{eq:3.26}--\eqref{eq:3.33}, we obtain \eqref{eq:4.22}. Formulas \eqref{eq:4.23} and \eqref{eq:4.24} can be verified similarly.

For (2), we use the transformation laws of Jacobi theta-functions under the action of $T$ and Proposition 2.2 in \cite{HY}, which easily yield \eqref{eq:4.25}, \eqref{eq:4.26} and \eqref{eq:4.27}.
\end{proof}

\begin{lem}
  For any function $\mathcal{F}_\lambda, \ \ \lambda=\{1,2,3\}$, its modular transformation is holomorphic in $(t,\tau)\in\mathbf{R}\times\mathbf{H}$.
\end{lem}
\begin{proof}
Similar to Lemma 3.5.
\end{proof}

\textit{Proof of Theorem 4.1.} We prove that these $\mathcal{F}_\lambda$ are holomorphic on $\mathbf{C}\times\mathbf{H}$, which implies the rigidity of Theorem 4.1. Denote by $\mathcal{F}$ one of the functions $\{\mathcal{F}_1,\mathcal{F}_2,\mathcal{F}_3\}$. By \eqref{eq:4.19}, \eqref{eq:4.20} and \eqref{eq:4.21}, the possible poles of $\mathcal{F}(t,\tau)$ can be written in the form $t=\frac{k}{l}(c\tau+d)$ for integers $k,l,c,d$ with $(c, d)=1$.

We can always find integers $a,b$ such that $ad-bc=1$. Then the matrix $g_2=\left(\begin{array}{cc}
\ d & -b  \\
 -c  & a
\end{array}\right)\in SL_2(\mathbb{Z})$ induces an action
$$\mathcal{F}(g_1(t,\tau))=\mathcal{F}\left(\frac{t}{-c\tau+a},\frac{d\tau-b}{-c\tau+a}\right).$$
Now, if $t=\frac{k}{l}(c\tau+d)$ is a polar divisor of $\mathcal{F}(t,\tau)$, then one polar divisor of $\mathcal{F}(g_2(t,\tau))$ is given by
$$\frac{t}{-c\tau+a}=\frac{k}{l}\left(c\frac{d\tau-b}{-c\tau+a}+d\right),$$
which yields $t=\frac{k}{l}$.

By Lemma 4.4, we know that, up to a constant, $\mathcal{F}(g_1(t,\tau))$ is still one of $\{\mathcal{F}_1,\mathcal{F}_2,\mathcal{F}_3\}$. This contradicts Lemma 4.5, thus completing the proof of Theorem 4.1.

Let $M$ be a $(4k+1)$-dimensional spin$^c$ manifold. Set 
\begin{equation}
\begin{split}
    \overline{W}^*_1&(M,L,V)\\
    &=\left\{\widehat{A}(TM){\rm exp}(\frac{c}{2}){\rm ch}(\Theta^*(T_{\mathbf{C}}M,L_{\bf{R}}\otimes\mathbf{C})){\rm ch}(Q_1(V_{\mathbf{C}})){\rm ch}(Q_1(E),g^{Q_1(E)},d,\tau)\right\}^{4k+1},
\end{split}
\end{equation}
\begin{equation}
\begin{split}
    \overline{W}^*_2&(M,L,V)\\
    &=\left\{\widehat{A}(TM){\rm exp}(\frac{c}{2}){\rm ch}(\Theta^*(T_{\mathbf{C}}M,L_{\bf{R}}\otimes\mathbf{C})){\rm ch}(Q_2(V_{\mathbf{C}})){\rm ch}(Q_2(E),g^{Q_2(E)},d,\tau)\right\}^{4k+1},
\end{split}
\end{equation}
\begin{equation}
\begin{split}
    \overline{W}^*_3&(M,L,V)\\
    &=\left\{\widehat{A}(TM){\rm exp}(\frac{c}{2}){\rm ch}(\Theta^*(T_{\mathbf{C}}M,L_{\bf{R}}\otimes\mathbf{C})){\rm ch}(Q_3(V_{\mathbf{C}})){\rm ch}(Q_3(E),g^{Q_3(E)},d,\tau)\right\}^{4k+1}.
\end{split}
\end{equation}

Now assume that $M$ admits an $S^1$-action that lifts to an action on $L$. Moreover, assume that this action preserves the given spin$^c$-structure associated to $(M, L)$, as well as the metrics and connections involved. Let $V$ be equipped with an $S^1$-action that restricts on each $V$-fiber over $M$ to a linear action preserving that fiber. Then the associated twisted Dirac operator is $S^1$-equivariant. We then have the following rigidity theorems.

\begin{thm}
  For a $(4k+1)$-dimensional connected spin manifold with a non-trivial $S^1$-action, and assuming $g_*:\pi_1(M)\to\pi_1(SO(N))=\mathbb{Z}_2$ is trivial, if $$p_1(L)_{S^1}+p_1(V)_{S^1}=p_1(TM)_{S^1}$$ and $c_3(E,g,d)=0$, then the Toeplitz operators $$\mathcal{T}_c\otimes\Theta^*(T_{\mathbf{C}}M,L_{\bf{R}}\otimes\mathbf{C})\otimes Q_\lambda(V_{\mathbf{C}})\otimes(Q_\lambda(E),g^{Q_\lambda(E)}),~\lambda=1,2,3$$ are rigid.
\end{thm}

To prove this rigidity theorem, we first calculate the corresponding Lefschetz numbers.
\begin{prop}
 The Lefschetz numbers of $\mathcal{T}_c\otimes\Theta^*(T_{\mathbf{C}}M,L_{\bf{R}}\otimes\mathbf{C})\otimes Q_\lambda(V_{\mathbf{C}}),~\lambda=1,2,3$ are
\begin{equation}\label{eq:4.34}
  \begin{split}
  \overline{\mathcal{L}}^*_1(g;\tau)=&2^{\bar{l}-\bar{r}}\left(\frac{-\mathbf{i}}{\pi}\right)^{\bar{r}}
  \sum_{\alpha}\int_{F_\alpha}\left(\prod^{s}_{j=1}
  y_j\frac{\theta'(0,\tau)}{\theta(y_j,\tau)}\prod^{r}_{\beta=1}
  \prod^{\dim\mathbf{N}_\beta}_{i_\beta=1}\frac{\theta'(0,\tau)}{\theta(x_{\beta}^{i_\beta}+
  m_{\beta}t,\tau)}\right)\\
  &\cdot\frac{\sqrt{-1}\theta(u+\sigma t,\tau)}{\theta_1(0,\tau)\theta_2(0,\tau)
  \theta_3(0,\tau)}\cdot\prod_{\nu=0}^l\prod_{h_\nu=1}^{\dim V_\nu}\frac{\theta_1(z_\nu^{h_\nu}+n_\nu t,\tau)}{\theta_1(0,\tau)}{\rm ch}(Q_1(E),g^{Q_1(E)},d,\tau),
  \end{split}
\end{equation}
\begin{equation}\label{eq:4.35}
  \begin{split}
  \overline{\mathcal{L}}^*_2(g;\tau)=&-\left(\frac{\mathbf{i}}{2\pi}\right)^{\bar{r}}
  \sum_{\alpha}\int_{F_\alpha}\left(\prod^{s}_{j=1}
  y_j\frac{\theta'(0,\tau)}{\theta(y_j,\tau)}\prod^{r}_{\beta=1}
  \prod^{\dim\mathbf{N}_\beta}_{i_\beta=1}\frac{\theta'(0,\tau)}{\theta(x_{\beta}^{i_\beta}+
  m_{\beta}t,\tau)}\right)\\
  &\cdot\frac{\sqrt{-1}\theta(u+\sigma t,\tau)}{\theta_1(0,\tau)\theta_2(0,\tau)
  \theta_3(0,\tau)}\cdot\prod_{\nu=0}^l\prod_{h_\nu=1}^{\dim V_\nu}\frac{\theta_2(z_\nu^{h_\nu}+n_\nu t,\tau)}{\theta_2(0,\tau)}{\rm ch}(Q_2(E),g^{Q_1(E)},d,\tau),
  \end{split}
\end{equation}
\begin{equation}\label{eq:4.36}
  \begin{split}
  \overline{\mathcal{L}}^*_3(g;\tau)=&-\left(\frac{\mathbf{i}}{2\pi}\right)^{\bar{r}}
  \sum_{\alpha}\int_{F_\alpha}\left(\prod^{s}_{j=1}
  y_j\frac{\theta'(0,\tau)}{\theta(y_j,\tau)}\prod^{r}_{\beta=1}
  \prod^{\dim\mathbf{N}_\beta}_{i_\beta=1}\frac{\theta'(0,\tau)}{\theta(x_{\beta}^{i_\beta}+
  m_{\beta}t,\tau)}\right)\\
  &\cdot\frac{\sqrt{-1}\theta(u+\sigma t,\tau)}{\theta_1(0,\tau)\theta_2(0,\tau)
  \theta_3(0,\tau)}\cdot\prod_{\nu=0}^l\prod_{h_\nu=1}^{\dim V_\nu}\frac{\theta_3(z_\nu^{h_\nu}+n_\nu t,\tau)}{\theta_3(0,\tau)}{\rm ch}(Q_3(E),g^{Q_1(E)},d,\tau).
  \end{split}
\end{equation}
\end{prop}
\begin{proof}
Similar to Proposition 3.7, we can calculate equations \ref{eq:4.34}, \ref{eq:4.35} and \ref{eq:4.36} by using the Lefschetz fixed point formula (cf. \cite{Wa}), (4.11) and (4.12).
\end{proof}

Next, we view the above expressions as defining functions $\mathcal{F}^*_\lambda(t,\tau)$ for each $\lambda\in\{1,2,3\}$, i.e., $\mathcal{F}^*_\lambda(t,\tau)=\overline{\mathcal{L}}^*_\lambda(g;\tau)$. Similar to Lemma 4.3, we have the following lemma.

\begin{lem}
Let $(t,\tau)\in\mathbf{C}\times\mathbf{H}$ be in the domain of $\mathcal{F}^*_\lambda(t,\tau)$.\\
(1) Then $\mathcal{F}^*_\lambda(t,\tau)=\mathcal{F}^*_\lambda(t+a,\tau)$ for any $a\in 2\mathbb{Z}$.\\
(2) If $p_1(L)_{S^1}+p_1(V)_{S^1}=p_1(TM)_{S^1}$, then $\mathcal{F}^*_\lambda(t,\tau)=\mathcal{F}^*_\lambda(t+a\tau,\tau)$ for any $a\in 2\mathbb{Z}$.
\end{lem}
\begin{proof}
Similar to Lemma \ref{le3.3}.
\end{proof}

Next, we prove that these $\mathcal{F}^*_\lambda$ are holomorphic in $(t,\tau)$ on $\mathbf{C}\times\mathbf{H}$.
\begin{lem}
  (1) Under the assumptions $p_1(L)_{S^1}+p_1(V)_{S^1}=p_1(TM)_{S^1}$ and $c_3(E,g,d)=0$, the action of $S$ on $\mathcal{F}_\lambda$ yields
  \begin{equation}\label{eq:4.22}
    \mathcal{F}^*_1\left(\frac{t}{\tau},-\frac{1}{\tau}\right)=2^{\bar{l}-\bar{r}+N/2}\tau^{2k}
    \mathcal{F}^*_2(t,\tau),
  \end{equation}
  \begin{equation}\label{eq:4.23}
    \mathcal{F}^*_2\left(\frac{t}{\tau},-\frac{1}{\tau}\right)=2^{-(\bar{l}-\bar{r}+N/2)}\tau^{2k}
    \mathcal{F}^*_1(t,\tau),
  \end{equation}
  \begin{equation}\label{eq:4.24}
    \mathcal{F}^*_3\left(\frac{t}{\tau},-\frac{1}{\tau}\right)=\tau^{2k}\mathcal{F}^*_3(t,\tau).
  \end{equation}
  (2) Under the action of $T$ on $\mathcal{F}_\lambda$,
  \begin{equation}\label{eq:4.25}
    \mathcal{F}^*_1(t,\tau+1)=\mathcal{F}^*_1(t,\tau),
  \end{equation}
   \begin{equation}\label{eq:4.26}
    \mathcal{F}^*_2(t,\tau+1)=\mathcal{F}^*_3(t,\tau),
  \end{equation}
   \begin{equation}\label{eq:4.27}
    \mathcal{F}^*_3(t,\tau+1)=\mathcal{F}^*_2(t,\tau).
  \end{equation}
\end{lem}
\begin{proof}
Similar to Lemma 4.4.
\end{proof}

\begin{lem}
  For any function $\mathcal{F}^*_\lambda, \ \ \lambda=\{1,2,3\}$, its modular transformation is holomorphic in $(t,\tau)\in\mathbf{R}\times\mathbf{H}$.
\end{lem}
\begin{proof}
Similar to Lemma 3.5.
\end{proof}

Next, similar to the proof of Theorem 4.1, we use Liu's method \cite{Li2} to prove Theorem 4.6 via Proposition 4.7, Lemma 4.8 and Lemma 4.9.

In the odd-dimensional case, using Liu's method \cite{Li2} and Han and Yu's method \cite{HY}, we obtain the following theorem.
\begin{thm}
 (i) For a $(4k-1)$-dimensional connected spin$^c$ manifold with a non-trivial $S^1$-action, and assuming $g_*:\pi_1(M)\to\pi_1(SO(N))=\mathbb{Z}_2$ is trivial, if $$3p_1(L)_{S^1}+3p_1(V)_{S^1}=p_1(TM)_{S^1}$$ and $c_3(E,g,d)=0$, then the twisted Toeplitz operators $$\mathcal{T}_c\otimes\Theta(T_{\mathbf{C}}M,L_{\bf{R}}\otimes\mathbf{C})\otimes Q(V_{\mathbf{C}})\otimes(Q_\lambda(E),g^{Q_\lambda(E)}),~\lambda=1,2,3$$ are rigid.
 
 (ii) For a $(4k+1)$-dimensional connected spin$^c$ manifold with a non-trivial $S^1$-action, and assuming $g_*:\pi_1(M)\to\pi_1(SO(N))=\mathbb{Z}_2$ is trivial, if $p_1(L)_{S^1}+3p_1(V)_{S^1}=p_1(TM)_{S^1}$ and $c_3(E,g,d)=0$, then the twisted Toeplitz operators $$\mathcal{T}_c\otimes\Theta^*(T_{\mathbf{C}}M,L_{\bf{R}}\otimes\mathbf{C})\otimes Q(V_{\mathbf{C}})\otimes(Q_\lambda(E),g^{Q_\lambda(E)}),~\lambda=1,2,3$$ are rigid.
 Where $Q(V_{\mathbf{C}})=Q_1(V_{\mathbf{C}})\otimes Q_2(V_{\mathbf{C}})\otimes Q_3(V_{\mathbf{C}})$.
\end{thm}

Similarly, consider the virtual vector bundle
$$Q(E)=Q_1(E_1)\otimes Q_2(E_2)\otimes Q_3(E_3).$$
Then we have
\begin{equation}
{\rm ch}(Q(E),g^{Q(E)},d,\tau)=-\frac{2^{\frac{N}{2}}}{8\pi^2}\int^1_0{\rm Tr}\left[g^{-1}dg\left(A\right)\right]du,
\end{equation}
with
$$A=\frac{\theta'_1(R_{\bar{u}}/(4\pi^2),\tau)}{\theta_1(R_{\bar{u}}/(4\pi^2),\tau)}
+\frac{\theta'_2(R_{\bar{u}}/(4\pi^2),\tau)}{\theta_2(R_{\bar{u}}/(4\pi^2),\tau)}+
\frac{\theta'_3(R_{\bar{u}}/(4\pi^2),\tau)}{\theta_3(R_{\bar{u}}/(4\pi^2),\tau)}.$$ 

Using Liu's method \cite{Li2} and Han and Yu's method \cite{HY}, we obtain the following theorem in the odd-dimensional case.
\begin{thm}
 (i) For a $(4k-1)$-dimensional connected spin$^c$ manifold with a non-trivial $S^1$-action, and assuming $g_*:\pi_1(M)\to\pi_1(SO(N))=\mathbb{Z}_2$ is trivial,\\
  1) if $3p_1(L)_{S^1}+p_1(V)_{S^1}=p_1(TM)_{S^1}$ and $c_3(E,g,d)=0$, then the twisted Toeplitz operators 
  $$\mathcal{T}_c\otimes\Theta(T_{\mathbf{C}}M,L_{\bf{R}}\otimes\mathbf{C})\otimes Q_\lambda(V_{\mathbf{C}})\otimes(Q(E),g^{Q(E)}),~\lambda=1,2,3$$ are rigid;\\
  2) if $3p_1(L)_{S^1}+3p_1(V)_{S^1}=p_1(TM)_{S^1}$ and $c_3(E,g,d)=0$, then the twisted Toeplitz operator 
  $$\mathcal{T}_c\otimes\Theta(T_{\mathbf{C}}M,L_{\bf{R}}\otimes\mathbf{C})\otimes Q(V_{\mathbf{C}})\otimes(Q(E),g^{Q(E)})$$ is rigid.
 
 (ii) For a $(4k+1)$-dimensional connected spin$^c$ manifold with a non-trivial $S^1$-action, and assuming $g_*:\pi_1(M)\to\pi_1(SO(N))=\mathbb{Z}_2$ is trivial,\\
  1) if $p_1(L)_{S^1}+p_1(V)_{S^1}=p_1(TM)_{S^1}$ and $c_3(E,g,d)=0$, then the twisted Toeplitz operators 
  $$\mathcal{T}_c\otimes\Theta^*(T_{\mathbf{C}}M,L_{\bf{R}}\otimes\mathbf{C})\otimes Q_\lambda(V_{\mathbf{C}})\otimes(Q(E),g^{Q(E)}),~\lambda=1,2,3$$ are rigid;\\
  2) if $p_1(L)_{S^1}+3p_1(V)_{S^1}=p_1(TM)_{S^1}$ and $c_3(E,g,d)=0$, then the twisted Toeplitz operator $\mathcal{T}_c\otimes\Theta^*(T_{\mathbf{C}}M,L_{\bf{R}}\otimes\mathbf{C})\otimes Q(V_{\mathbf{C}})\otimes(Q(E),g^{Q(E)})$ is rigid.
\end{thm}

\section{Acknowledgements}

 This work was supported by the National Key Research and Development Program of China (NKPs) [Grant Number 2024YFA1013201], the Natural Science Foundation of Chongqing (NSFCQ) [Grant Number CSTB2024NSCQ-LZX0040], the Special Project of Chongqing Municipal Science and Technology Bureau [Grant Number 2025CCZ015], and supported in part by NSFC No.11771070. The third author is the corresponding author of this article. The authors also thank the referee for his (or her) careful reading and helpful comments.

\vskip 1 true cm

\section{Data availability}

No data was gathered for this article.

\section{Conflict of interest}

The authors have no relevant financial or non-financial interests to disclose.

\vskip 1 true cm

\bigskip
\bigskip
\indent{J. Guan}\\
 \indent{School of Mathematics and Statistics,
Northeast Normal University, Changchun Jilin, 130024, China }\\
\indent E-mail: {\it guanjy@nenu.edu.cn }\\
\indent{K. Liu}\\
 \indent{Mathematical Science Research Center,Chongqing University of Technology, Chongqing, 400054, China }\\
\indent E-mail: {\it kefengliu@cqut.edu.cn }\\
\indent{Y. Wang}\\
 \indent{School of Mathematics and Statistics,
Northeast Normal University, Changchun Jilin, 130024, China }\\
\indent E-mail: {\it wangy581@nenu.edu.cn }\\

\end{document}